\documentclass{amsart}

\usepackage[all]{xy}
\usepackage{amssymb}

\newcommand{\palpha}{{\overline \alpha}} 
\newcommand{\M}{{\mathrm {M}}}
\renewcommand{\u}{{\mathrm {u}}}
\newcommand{\Map}{{\mathrm {Map}}}
\newcommand{\Witt}{{\mathrm {W}}}
\newcommand{\Aut}{{\mathrm {Aut}}}
\newcommand{\End}{{\mathrm {End}}}
\newcommand{\Hom}{{\mathrm {Hom}}}
\newcommand{\Gl}{{\mathrm {GL}}}
\newcommand{\GL}{{\mathrm {GL}}}
 \newcommand{\PGL}{{\mathrm{PGL}}}
\newcommand{\Def}{{\mathrm {Def}}}
\newcommand{\Z}{{\mathbb {Z}}}
\newcommand{\F}{{\mathbb {F}}}
\newcommand{\gm}{{\mathfrak{m}}}

\theoremstyle{plain}
\newtheorem{thm}{Theorem}[section]

\newtheorem{lemma}[thm]{Lemma}

\newtheorem{question}[thm]{Question}

\newtheorem{hypo}[thm]{Hypothesis}

\theoremstyle{definition}

\theoremstyle{remark}
\newtheorem{rem}[thm]{Remark}

\newcommand{\semid}{\unitlength.47cm
 \begin{picture}(.7,.6)
  \put(0,.05){$\times$}
  \put(.47,.04){\line(0,1){.39}}
 \end{picture}}

\begin{document}

\title{Inverse Problems for deformation rings}

\date{}
\author{Frauke M. Bleher}
\address{F.B.: Department of Mathematics\\University of Iowa\\
Iowa City, IA 52242-1419, U.S.A.}
\email{frauke-bleher@uiowa.edu}
\author{Ted Chinburg}
\address{T.C.: Department of Mathematics\\University of
Pennsylvania\\Phi\-la\-delphia, PA
19104-6395, U.S.A.}
\email{ted@math.upenn.edu}
\author{Bart de Smit}
\address{B.deS:  Mathematisch Instituut\\University of Leiden\\P.O. Box 9512
\\ 2300 RA Leiden\\
The Netherlands}
\email{desmit@math.leidenuniv.nl}

\thanks{The first author was supported in part by  
NSF Grant  DMS0651332 and NSA Grant H98230-11-1-0131.
The second author was supported in part by  NSF Grants DMS0801030
and DMS1100355.
The third author was funded in part by the European Commission under contract
MRTN-CT-2006-035495.}

\begin{abstract}
Let $W$ be a complete Noetherian local commutative ring with
residue field $k$ of positive characteristic $p$.
We study the inverse problem for the universal deformation rings $R_{W}(\Gamma,V)$ 
relative to $W$ of finite dimensional representations $V$
of a profinite group $\Gamma$ over $k$.
We show that for all $p$ and $n \ge 1$, the
ring  $W[[t]]/(p^n t,t^2)$ arises as a universal deformation ring. 
This ring is not a complete intersection if $p^nW\neq\{0\}$, so we obtain an answer to a question of M. Flach
in all characteristics. We also study the `inverse inverse problem' for the ring $W[[t]]/(p^n t,t^2)$; this is to determine
all pairs $(\Gamma, V)$ such that $R_{W}(\Gamma,V)$ is isomorphic to this ring.
\end{abstract}

\maketitle

\section{Introduction}
\label{s:intro}

Let $W$ be a complete Noetherian local commutative ring with
residue field $k$ of positive characteristic $p$. 
Suppose $\Gamma$ is a profinite group and that $V$ is a continuous finite
dimensional representation of $\Gamma$ over $k$.
Here the topology on $V$ is discrete, so the image of the 
continuous homomorphism $\Gamma \to \Aut_k(V)$ is finite.
In \S \ref{s:basechange} we recall the definition of a deformation of $V$ over
a complete Noetherian local commutative $W$-algebra with residue
field $k$.
Under a mild hypothesis on the representation ($\End_{k\Gamma}(V)=k$)
and on the profinite group $\Gamma$ (Hypothesis \ref{hyp:finite} below), 
there is a unique universal deformation over the so-called
universal deformation ring $R_W(\Gamma, V)$. We will recall its basic properties in the next section.

In this paper we consider the following inverse problem:

\begin{question}
\label{q:inverse}
Which complete noetherian local $W$-algebras $R$ with residue field $k$ are 
isomorphic to $R_{W}(\Gamma,V)$ for some $\Gamma$ and $V$ as above? 
\end{question}

It is important to emphasize that in this question $W$ is fixed, but $\Gamma$ and $V$ are not fixed. 
Thus for a given $W$-algebra $R$, one would like to construct both a profinite group
$\Gamma$ and a continuous finite dimensional representation $V$ of $\Gamma$
over $k$ for which $R_{W}(\Gamma,V)$ is isomorphic to $R$.

One can also consider the following ``inverse inverse'' problem:

\begin{question}
\label{q:inverseinverse}
Suppose $R$ is a complete Noetherian local ${W}$-algebra with residue
field $k$. What are all pairs $(\Gamma, V)$ as above such that
$R\cong R_{W}(\Gamma,V)$?
\end{question}

The goal of this paper is to answer Questions \ref{q:inverse} and \ref{q:inverseinverse}
for the rings $R=W[[t]]/(p^n t,t^2)$, where $n$ is a positive integer.
More precisely, we prove the following main
results.

\begin{thm}
\label{thm:main1}
For every $n \ge 1$ there is a representation $V$ of a finite group $\Gamma$
over $k$ such that $R_{W}(\Gamma,V)$ is isomorphic to $$W[[t]]/(p^n t,t^2)$$
as a $W$-algebra.
If $p^nW\ne 0$ then this ring is not a complete intersection.
\end{thm}

Note that $W[[t]]/(p^nt,t^2)$ is not even a Cohen-Macaulay ring in general.

\begin{thm}
\label{thm:main2}
Suppose that $k$ is perfect and let $\Witt(k)$ be the ring of infinite 
$p$-Witt vectors over~$k$.
Then there exists a complete classification, given in Theorem \ref{thm:genresult}, of
all profinite groups $\Gamma$ and all continuous finite dimensional 
representations $V$ of $\Gamma$ over $k$ such that the following conditions hold:
\begin{enumerate}
\item $\End_{k\Gamma}(V)=k$ and Hypothesis \ref{hyp:finite} holds;
\item $V$ is projective as a module over
$kG$, where $G$ denotes the image of $\Gamma$ in $\Aut_k(V)$;
\item the universal deformation of $V$ is faithful as a representation of $\Gamma$;
\item the universal deformation ring $R_{W}(\Gamma,V)$ is isomorphic to 
$$\Witt(k)[[t]]/(p^nt,t^2).$$
\end{enumerate}
\end{thm}

There is an extensive literature concerning explicit computations of universal
deformation rings (often with additional deformation conditions).  In
\cite{boe}, B\"ockle gives a survey of recent results on presentations of
deformation rings and of applications of such presentations to arithmetic
geometry.  This includes many example of rings which are known to be
deformation rings.  There is also a discussion in \cite{boe} of how one can
show that deformation rings are complete intersections as well as the relevance
of presentations to arithmetic, e.g. to  Serre's conjectures in the theory of
modular forms and Galois representations.  

The problem of constructing  representations having universal deformation rings
which are not complete intersections was first posed by M. Flach
\cite{flach}.  The first example of a representation of this kind was found by
Bleher and Chinburg when $\mathrm{char}(k) = 2$;  see \cite{lcicomptes,lcann}.
A more elementary argument proving the same result  was given in \cite{JB}.
Theorem \ref{thm:main1} gives an answer to Flach's question for all possible
residue fields of positive characteristic.  

As of this writing we do not know of a complete local commutative Noetherian
ring $R$ with perfect residue field $k$ of positive characteristic which cannot
be realized as a universal deformation ring of the form $R_{\Witt(k)}(\Gamma,V)$
for some profinite $\Gamma$ and some representation $V$ of $\Gamma$ over $k$.  

Theorem \ref{thm:main1} and the formulation of the inverse problem in Question
\ref{q:inverse} first appeared in \cite{bcd}.  In subsequent work on the
inverse problem, Rainone found in \cite{R} some other rings which are universal
deformation rings and not complete intersections;  see Remark
\ref{rem:negative}.

The sections of this paper are as follows.  

In \S \ref{s:basechange} we recall the definitions of deformations and of
versal and universal deformation rings and describe how versal deformation
rings change when extending the residue field $k$ (see Theorem
\ref{thm:basechange}). 
  
In \S \ref{s:computit} we consider arbitrary perfect fields $k$ of
characteristic $p$ and we take $W = \Witt(k)$.  In Theorem \ref{thm:genresult},
which implies Theorem \ref{thm:main2}, we give a sufficient and necessary set
of conditions on a representation $\tilde {V}$ of a finite group $\Gamma$ over
$k$ for the universal deformation ring $R_{\Witt(k)}(\Gamma,\tilde{V})$ to be
isomorphic to $R = \Witt(k)[[t]]/(p^n t,t^2)$.  The proof that these conditions
are sufficient involves first showing that $R_{\Witt(k)}(\Gamma,\tilde{V})$ is
a quotient of $\Witt(k)[[t]]$ by proving that the dimension of the tangent
space of the deformation functor associated to $\tilde{V}$ is one.  We then
construct an explicit lift of $\tilde{V}$ over $R$ and show that this cannot be
lifted further to any small extension ring of $R$ which is a quotient of
$\Witt(k)[[t]]$. 
 
In \S\ref{s:inverse} we prove Theorem \ref{thm:main1}.  We use Theorem
\ref{thm:basechange} to reduce the proof of Theorem \ref{thm:main1}  to the
case in which  $k = \mathbb{F}_p=\mathbb{Z}/p$ and $W  = \Witt(k) =
\mathbb{Z}_p$. In the latter case we provide explicit examples using
Theorem \ref{thm:genresult} and twisted
group algebras of the form $\mathbb{F}_{p^2}^{\, \dagger}G_0$ where 
$G_0=\mathrm{Gal}(\mathbb{F}_{p^2}/\mathbb{F}_p)$.
Some further examples illustrating Theorem \ref{thm:main1} are given in 
\S\ref{s:examples}.
 
\medbreak \noindent {\bf Acknowledgments:} The authors would like to thank M.
Flach for correspondence about his question.  The second author would also like
to thank the University of Leiden for its hospitality during the spring of 2009
and the summer of 2010.

\section{Deformation rings}
\label{s:basechange}

Let $\Gamma$ be a profinite group, and let $k$ be a field of characteristic $p > 0$,
and let $W$ be a complete 
Noetherian local ring with residue field $k$.  
\newcommand\cat{\hat{\mathcal{C}}}
We denote by $\cat$ the category of all 
complete Noetherian local $W$-algebras with residue field $k$.  
Morphisms in $\hat{\mathcal{C}}$ are $W$-algebra homomorphisms --- they are
continuous and they induce the identity map on residue fields.

Let $V$ be a finite dimensional vector space over $k$ that is endowed with a
continuous $k$-linear action of $\Gamma$, i.e., a continuous group homomorphism
from $\Gamma$ to the discrete group $\Aut_k(V)$.
A \emph{lift} of $V$ over a ring $A \in \hat{\mathcal{C}}$ is then a pair $(M,\phi)$
consisting of a finitely generated free $A$-module $M$ on which $\Gamma$ acts
continuously together with a $\Gamma$-equivariant isomorphism $\phi\colon\;k \otimes_A M \to V$ of
$k$-vector spaces.  
We define the set $\mathrm{Def}_V(A)$ of \emph{deformations} of $V$ over $A$ to be the
set of isomorphism classes of lifts of $V$ over $A$.  
If $\alpha:A\to A'$ is a morphism in $\hat{\mathcal{C}}$, then we define the map
$$\mathrm{Def}_V(\alpha)\colon\;\mathrm{Def}_V(A)\to \mathrm{Def}_V(A')\quad\mbox{  by }
\quad[M,\phi]\mapsto [A'\otimes_{A,\alpha}M,\phi_\alpha]$$ where 
$\phi_\alpha$ is the composition 
$$k\otimes_{A'}(A'\otimes_{A,\alpha}M)\cong k\otimes_AM\xrightarrow{\phi} V.$$
With these definitions $\mathrm{Def}_V(-)$ is a functor 
from $\hat{\mathcal{C}}$ to the category of sets.

We can describe $\mathrm{Def}_V(A)$ in terms of matrix groups as follows.
By choosing a $k$-basis of $V$ we can identify $V$ with $k^d$.
The $\Gamma$-action on $V$ is then given by
a continuous homomorphism $\rho:\Gamma \to \mathrm{GL}_d(k)$.
Let $A$ be a ring in $\hat{\mathcal{C}}$ and denote the reduction map
$\mathrm{GL}_d(A)\to \mathrm{GL}_d(k)$ by $\pi_A$.
By a \emph{lift} of $\rho$ over a ring $A$ in $\hat{\mathcal{C}}$ we mean a continuous
homomorphism $\tau:\Gamma\to \mathrm{GL}_d(A)$ such that $\pi_A\circ\tau=\rho$.
Such a lift defines a $\Gamma$-action on $M=A^d$ and with the obvious
isomorphism $\phi:M\otimes_Ak\to V$ such a lift defines a deformation
$[M,\phi]$ of $V$ over $A$.
Two lifts $\tau,\tau':\Gamma\to \mathrm{GL}_d(A)$ of $\rho$ over $A$ give rise
to the same deformation if and only if they are are 
\emph{strictly equivalent},
that is, if one can be brought into the other by conjugation by a matrix in the kernel of $\pi_A$.
In this way the choice of a basis of $V$ gives rise to an identification of
$\mathrm{Def}_V(A)$ with the set $\Def_\rho(A)$ of strict equivalence
classes of lifts of $\rho$ over $A$.

The functor $\mathrm{Def}_V(-)=\Def_\rho(-)$ is representable if there is a ring $R$
in $\hat{\mathcal{C}}$, and a lift $(U,\phi_U)$ of $V$ over $R$
so for all $A$ in $\hat{\mathcal{C}}$ the map
$$
f_A: \mathrm{Hom}_{\hat{\mathcal{C}}}(R,A)\to \mathrm{Def}_V(A) \qquad \alpha \mapsto \mathrm{Def}_V(\alpha)([U,\phi_U])
$$
is bijective. If this is the case then $R$ is said to be the universal
deformation ring of $V$ and we write $R=R_{W}(\Gamma,V)=R_{W}(\Gamma,\rho)$.
The defining property determines $R$ uniquely up to a unique isomorphism.

A slightly weaker notion can be useful if the functor $\mathrm{Def}_V(-)$ is
not representable.  The ring $k[\epsilon]$ of dual numbers with $\epsilon^2 = 0$
has $W$-algebra structure such that the maximal ideal of $W$ annihilates
$k[\epsilon]$.  One says $R$ is a versal deformation ring of $V$ if
the maps $f_A$ are surjective for all $A$, and bijective
for $A=k[\epsilon]$.  These conditions determine $R=R_{W}(\Gamma,V)=R_{W}(\Gamma,\rho)$ uniquely up to
isomorphism, but the isomorphism need not be unique.
By Mazur \cite[Prop. 20.1]{maz2} the functor $\mathrm{Def}_V(-)$ is
continuous, which means that one only needs to check the surjectivity of $f_A$
for Artinian~$A$.

We will suppose from now on that $\Gamma$ satisfies the following
$p$-finiteness condition used by Mazur in \cite[\S 1.1]{maz1}:

\begin{hypo}
\label{hyp:finite} 
For every open subgroup $J$ of finite index in $\Gamma$, there
are only finitely many continuous homomorphisms from $J$ to $\mathbb{Z}/p$.
\end{hypo}

It follows by \cite[\S1.2]{maz1} that for $\Gamma$ satisfying Hypothesis \ref{hyp:finite}, all finite 
dimensional continuous representations $V$ of 
$\Gamma$ over $k$ have a versal deformation ring.  It is shown in  \cite[Prop. 7.1]{desmitlenstra} that
if $\mathrm{End}_{k\Gamma}(V)=k$, then $V$ has a universal deformation ring.

A proof of the following base change result is given in an appendix (see \S\ref{s:appendix}).
For finite extensions of $k$, this was proved by Faltings (see \cite[Ch. 1]{Wiles}).

\begin{thm}
\label{thm:basechange}
Let $\Gamma$, $k$, $W$ and $V$ be as above.
Suppose that we have a local homomorphism $W\to W'$ of
complete Noetherian local rings, and denote the residue field of $W'$
by $k'$.
Then the versal deformation ring $R_{W'}(\Gamma,V\otimes_kk' )$ is
the completion of the local ring $W'\otimes_{W}R_{W}(\Gamma,V)$.
\end{thm}

\section{The inverse inverse problem for $R=\Witt(k)[[t]]/(p^nt,t^2)$}
\label{s:computit} 
  
In this section $k$ denotes a perfect field $k$ of positive characteristic $p$,
and $W=\Witt(k)$ is the ring of infinite $p$-Witt vectors over $k$, which is a
complete discrete valuation ring of characteristic $0$, residue field $k$, and
uniformizer $p$.  

We let $\Gamma$ be a profinite group satisfying Hypothesis \ref{hyp:finite},
and we let $V$ be a finite dimensional continuous representation of
$\Gamma$ over $k$ that satisfies $\End_{k\Gamma}(V)=k$.
By the previous section we then know that universal deformation
ring $R_W(\Gamma, V)$ exists.

We let $K$ be the kernel of the group homomorphism $\Gamma\to \Aut_k(V)$ and
we let $G$ be the image. Then $G$ is a finite group, and we make the
additional assumption that $V$ is projective as a $kG$-module.
The group $\Gamma$ acts by conjugation on $K$, and if $K$ is abelian then
this makes $K$ into a $\Z G$-module.

Since $V$ is a projective $kG$-module there is a
unique deformation $[V_W,\varphi]$ of the $kG$ module $V$ to $W$ such that $V_W$
is projective as a $WG$-module \cite[Prop. 42, \S 14.4]{SerreRep}. Note that $\varphi$ then provides an
identification of $V_W/pV_W=V_W\otimes_Wk$ with $V$. 
Note also that $M_W=\End_W(V_W)$ has the structure of a (possibly non-commutative) 
$W$-algebra by composition of endomorphisms. The natural $G$-action on $M_W$ by
conjugation also makes it into a $WG$ module, and since $V_W$ is a projective
$WG$-module, so is $M_W$.
In the same way the $k$-algebra $M=\End_k(V)=M_W/pM_W$ is a projective
$kG$-module.

\begin{thm}
\label{thm:genresult}    
For all $k,\Gamma,V$ as above, and all $n\ge 1$ the following statements (i) and (ii) are equivalent.
\begin{enumerate}
\item[(i)]
The group $\Gamma$ acts faithfully on the universal deformation of $V$,
and the universal deformation ring $R_W(\Gamma,V)$ is isomorphic to $W[[t]]/(p^nt,t^2)$.
\item[(ii)] The following conditions hold:
\begin{enumerate}
\item[(a)]  the group 
$\mathrm{Hom}^{cont}(K , M)^G$ of continuous $G$-equivariant homomorphisms from $K$ to $M$ has dimension $1$ as a
vector space over $k$;
\item[(b)] there is a continuous injective $G$-equivariant homomorphism 
$$\alpha\colon\;K \to M_W/p^nM_W$$
such that 
\begin{itemize}
\item 
there exist $g,h\in K$ with 
$$\alpha(g)\circ\alpha(h) \not\equiv \alpha(h)\circ\alpha(g) \bmod pM_W/p^nM_W,$$ or
\item $n=1$, $p = 2$ and there is no $a\in k$ such that for all
$g\in K$ we have $\alpha(g)^2= a\, \alpha(g)$.
\end{itemize}
\end{enumerate}
\end{enumerate}
Each of conditions (i) and (ii) implies that the group $\Gamma$ is finite and $K$ is abelian.
\end{thm}

Note that Theorem \ref{thm:genresult} implies Theorem \ref{thm:main2}. To show
Theorem \ref{thm:main1}, we construct in Section \ref{s:inverse} examples for
which the conditions in Theorem  \ref{thm:genresult}(ii) are satisfied.  The
rest of this Section is devoted to the proof of Theorem \ref{thm:genresult}.
We first establish some basic cohomological results that will be used in the proof.

\begin{lemma} With the assumptions prior to the statement of Theorem \ref{thm:genresult}  the following conditions hold:
\label{lem:cohom}
\begin{enumerate}
\item
$H^i(G,M_W/p^\ell M_W) = 0$ for all $i\ge 1$ and $\ell \ge 0$;
\item
$H^1(\Gamma,M)$ and $\mathrm{Hom}^{cont}(K , M)^G$ are isomorphic vector spaces over $k$;
\item
the restriction map $H^2(\Gamma, M)\to H^2(K,M)$ is injective.
\end{enumerate}
\end{lemma}

\begin{proof} 
To see $(1)$, note that $M_W/p^\ell M_W$ is a projective module over the group ring
$(W/p^\ell W) G$, so it is cohomologically trivial as a $\Z G$-module.

The inflation-restriction sequence for $M=M_W/pM_W$ now gives
an exact sequence 
$$
\xymatrix{
0=H^1(G,M) \ar[r] &H^1(\Gamma, M) \ar[r] &H^1(K,M)^G \ar[r] &H^2(G,M) =0.}
$$
Using the equalities $H^1(K,M)^G=\Hom^{cont}(K,M)^G$ we see that (2) holds.

It remains to show (3). Since $M$ is a finitely generated projective $kG$-module,
$\mathrm{Hom}^{cont}(K,M)$ is isomorphic to a direct summand of a $kG$-module
that is induced from the trivial subgroup of $G$, so
$H^1(K,M)=\mathrm{Hom}^{cont}(K,M)$ is a cohomologically trivial $kG$-module.
This implies that the Lyndon/Hochschild-Serre spectral sequence for $H^2(\Gamma,M)$ 
degenerates, and we get isomorphisms
\begin{eqnarray*}
H^2(\Gamma,{M}) &\cong& H^0(G,H^2(K,{M})) \\
&\cong& H^2(K,{M})^G\; \subset \; H^2(K,M),
\end{eqnarray*}
where the composite map is the restriction map \cite[Theorem 6.8.2]{Weibel}. 
\end{proof}

\begin{lemma}
\label{lem:tangent}
Condition (ii)(a) in Theorem \ref{thm:genresult}    
is equivalent to the condition that $R_W(\Gamma,V)$ is of the form
$W[[t]]/I$ where the ideal $I$ is contained in $\gm=(p,t^2)$.
\end{lemma}
\begin{proof}
For a ring $A$ in $\cat$ we denote its maximal ideal by $\gm_A$.
The cotangent space $t^*_A$ of $A$ is the vector space 
$\gm_A/(\gm_A^2+pA)$ over $k$. Recall that 
a morphism $A\to B$ in $\cat$ is surjective if and only if
the induced map $t_A^*\to t_B^*$ is surjective, and that 
$\dim_kt_A^*$ is the minimal number $s$ so that $A$ is isomorphic to
a quotient of $W[[t_1,\ldots,t_s]]$.
For $R=R_{W}(\Gamma,V)$ the tangent space $t_R=\Hom_k(t_R^*,k)$ is
naturally isomorphic to $H^1(\Gamma,M)$ \cite[p. 391]{maz1}. Combining this
with the previous lemma we see that condition (ii)(a) in Theorem \ref{thm:genresult}
is equivalent to $\dim_k t_R=1$. Rings of the form stated clearly have 
1-dimensional cotangent spaces. Conversely, if $t_R^*$ is $1$-dimensional
then any morphism $W[[t]]\to R$ in $\cat$ sending $t$ to an element of $\gm_R$ which
is not zero in $t_R^*$, is surjective and its kernel is contained in
$(p,t^2)$.
\end{proof}

\subsection{Proof that (ii) implies (i) in Theorem \ref{thm:genresult}}
Throughout this subsection, we assume that
condition (ii) of Theorem \ref{thm:genresult} holds.
Put $d=\dim_k(V)$, choose a basis of $V_W$ over $W$ and let the $G$-action
on $V_W$ be given by $\rho_W\colon\; G\to \GL_d(W)$.
We put 
\begin{eqnarray*}
R&=&W[[t]]/(p^nt,t^2)\\
&=& W\; \oplus\; (W/p^nW)t.
\end{eqnarray*}
The morphism $R\to W$ in $\cat$ sending $t$ to $0$ gives rise to an exact sequence of groups
\begin{equation*}
\xymatrix {
0\ar[r]&\M_d(W/p^nW) \ar[r]^-\varphi&\mathrm{GL}_d(R)\ar[r] &\mathrm{GL}_d(W)\ar[r]&1
}
\end{equation*}
where $\varphi(A)=1+tA$.  
Note that the action of $\GL_d(W)$ on $\M_d(W/p^nW)$ induced by the exact
sequence is given by the conjugation action on $\M_d(W)$, taken modulo $p^n$.
Using our basis of $V_W$ we can identify $\M_d(W/p^nW)$ with
$M_W/p^nM_W$ and this identification respects $G$-action. 
Thus, condition (ii) provides an injective homomorphism $\alpha$
in the following diagram.
\begin{equation}
\label{eq:lifter}
\xymatrix @R=3pc{
1\ar[r]&K \ar[d]^{\alpha} \ar[r]& \Gamma \ar @{-->}[d]^{\rho_R} \ar[r]&G\ar[d]^{\rho_W}\ar[r]&1\\
0\ar[r]&\M_d(W/p^nW) \ar[r]^-\varphi&\mathrm{GL}_d(R)\ar[r] &\mathrm{GL}_d(W)\ar[r]&1
}
\end{equation}
Note that the rows in this diagram are exact.  By \cite[p. 179]{ArtinTate}, the obstruction 
to the existence of a homomorphism $\rho_R$ which makes (\ref{eq:lifter}) commute lies
in the group  $H^2(G,\M_d(W/p^nW))$.  By Lemma \ref{lem:cohom} part (1),
this is a trivial group, so $\rho_R$ exists.  
Since $\alpha$ and $\rho_W$ are injective, the map $\rho_R$ is also
injective. Thus, $\Gamma$ acts faithfully on the deformation
over $R$ associated to $\rho_R$, so it certainly acts faithfully on the universal deformation.

Writing $R^\u=R(\Gamma,V)$ the universal property of deformation rings
provides us with a unique morphism $\gamma:R^\u\to R$ in $\cat$ and a lift
$\rho^\u\colon\; \Gamma\to \GL_d(R^\u)$ such that the composite map
$\gamma \circ \rho^\u:\Gamma\to \GL_d(R)$ is strictly equivalent to $\rho_R$.
We will show that  $\gamma$ is an isomorphism.

Consider the surjection $\beta:R \to R/pR \cong k[t]/(t^2)$. If $\gamma$ 
is not surjective, then $\beta \circ \gamma$ would be the composition
of the residue map $r:R^\u \to k$ with the natural $k$-algebra inclusion
$\iota:k \to k[t]/(t^2)$.  This would imply that 
$$\beta \circ \rho_R:\Gamma \to \GL_d(R/pR) =\GL_d(k[t]/(t^2))$$ is strictly equivalent to 
$\beta \circ \gamma \circ \rho^\u = \iota \circ r \circ \rho^\u$.  However,
by condition (ii)(b) there is an element $g\in K$ so that $\alpha(g)$
does not lie in $pM_W/ p^n M_W = p M_d(W/p^nW)$.   This $g$ is 
in the kernel of  $\iota \circ r \circ \rho^\u $ but not in the kernel of
$\beta \circ \rho_R$, so these representations cannot be strictly equivalent.
We conclude that 
$\gamma:R(\Gamma,V)\to R$ is surjective.  Thus we can and will replace
$\rho^\u$ by a strictly equivalent lift so that $\gamma \circ \rho^\u = \rho_R$.

Now define a morphism
$\pi\colon\; W[[t]]\to R^\u$ in $\cat$ by letting $\pi(t)$ be any
element in $R^\u$ that is mapped by $\gamma$ to $t\in R$.
By condition (ii)(a) and Lemma \ref{lem:tangent} we then see that
$\pi$ is surjective. In the chain of surjections 
$$\xymatrix{W[[t]]\ar[r] & R^\u\ar[r]& R \ar[r]& W}$$
we will write $t$ for the image of $t$ in any of the rings. So we have
$t=0$ in $W$, and $t^2=p^nt=0$ in $R$, and our aim is to show that
$t^2=p^nt=0$ in $R^\u$.  Note that 
$$R^\u/t R^\u = W\quad\mbox{ and }\quad R^\u/(tR^\u + p^n R^\u) = W/p^nW.$$

Suppose that we are in the first case of condition (b) of (ii),
that is, we have $$\alpha(g)\circ\alpha(h) \not\equiv \alpha(h)\circ\alpha(g)
\;\bmod pM_W/p^nM_W$$ for certain $g,h\in K$. Write
$\rho^\u(g)=1+tA$ and $\rho^\u(h)=1+tB$ with $A,B\in \M_d(R^\u)$
whose images modulo $(p^n,t)$ are $\alpha(g)$ and $\alpha(h)$.
Since $K$ is abelian 
we have 
$$(1+tA)(1+tB)=(1+tB)(1+tA),$$ so $t^2(AB-BA)=0$.
The matrix $AB-BA$ now has some entry which is a unit in $R^\u$.
Hence  $t^2=0$ in $R^\u$ and $A$ has an entry which is a unit in $R^\u$.
Our element $g\in K$ has order dividing  $p^n$.   Hence 
$$1=(1+tA)^{p^n}=1+p^nt A $$ so $p^nt A=0$. Since $A$  has an entry which is a unit,
it follows that $p^nt=0$ in $R^\u$.

Now suppose that we are in the second case of condition (b) of (ii),
so we have $n=1$ and $p=2$.
For each $g\in K$ write $\rho^\u(g)=1+tA_g$ with $A_g\in \M_d(R^\u)$
such that $A_g\equiv \alpha(g)\bmod (p,t)$ where $(p,t) = (2,t)$.
Since all $g\in K$ are annihilated by $2$ we have
$$
1=(1+tA_g)^2=1+2tA_g+t^2A_g^2.
$$
Let us now consider the free $R^\u$-module $P=\Map(K,M_d(R^\u))$.
This contains the elements $v_1\colon\;g\mapsto A_g$ and
$v_2\colon\;g\mapsto A_g^2$. Then on the one hand, the identity
above shows that $2tv_1+t^2v_2=0$. On the other hand
the second case of condition (b) of (ii) 
implies that the images of $v_1$ and $v_2$ in $P\otimes_{R^\u}k$
are linearly independent over $k$. By Nakayama's lemma
this means that $v_1$ and $v_2$ are part of a basis of $P$ over $R^\u$.
So it follows that $2t=t^2=0$ in $R^\u$.

\subsection{Proof that (i) implies (ii) in Theorem \ref{thm:genresult}}
Throughout this subsection, we assume that condition (i)
of Theorem \ref{thm:genresult} holds.
Property (ii)(a) of Theorem \ref{thm:genresult} follows
from Lemma \ref{lem:tangent}.

The action of $G$ on $V_W$
gives a morphism 
$$R=R(\Gamma,V)=W[[t]]/(p^nt,t^2)\to W.$$ 
There is only
one such morphism and it sends $t\in R$ to $0$. Since we are
assuming that $R$ is the universal deformation ring $R(\Gamma,V)$
and $R \to W$ is surjective,  we can now find maps $\rho_W$, $\rho_R$ and $\alpha$ as
in the commutative diagram (\ref{eq:lifter}).
The map $\rho_R$ is injective by assumption (i) so $\alpha$
is also injective, $G$-equivariant and continuous.  
Since
$\M_d(W/p^nW)$ is abelian and has the discrete topology, this implies
$K$ is finite and abelian, so $\Gamma$ is also finite.  

In order to prove (ii)(b) we now assume that (ii)(b) does not hold, and we will derive
a contradiction. The first step in doing this is to lift $\rho_R$ 
to a suitable extension $R'$ of $R$, i.e., 
to a ring $R'$ in $\cat$ with an ideal
$I$ such that $R'/I=R$. In all cases that we will consider we have 
$I^2=0$, so we have an exact sequence 
\begin{equation}
\label{eq:smalllift}
\xymatrix @R=3pc {
&& &\Gamma \ar[d]^{\rho_R} &\\
0\ar[r]&\M_d(I) \ar[r]^-\varphi&\mathrm{GL}_d(R')\ar[r]\ar @{<--}^{\rho'}[ur] &\mathrm{GL}_d(R)\ar[r]&1
}
\end{equation}
where $\varphi(A)=1+A$. We now show how to choose $R'$
and produce the restriction of the lift $\rho'$ to $K$.

Let us write $$\palpha\colon\; K\to M=M_d(k)$$ for the map 
sending $g\in K$ to $\alpha(g)\bmod pM_W/p^nM_W$.
Since we assumed that (ii)(b) does not hold, the set $\palpha(K)$
is contained in a commutative sub-$k$-algebra of $\M_d(k)$.

If $p\ne 2$ then we take 
\begin{eqnarray*}
R'&=&W[[t]]/(p^nt,pt^2,t^3)\\
&=& W\;\oplus\; (W/p^nW)t \;\oplus\; k t^2
\end{eqnarray*}
and we let $$\rho'(g)=1+\alpha(g)t +\palpha(g)^2 t^2/2$$ for $g\in K$.
Using that the elements $\palpha(K)$ commute in $M_d(k)$ one easily shows that
this exponential map is a homomorphism $K\to \GL_d(R')$.  

Now suppose that $p=2$ and $n\ge 2$.  We take $R'$ as above.
For $g\in K$ we now claim that $g$ and $1+t\alpha(g)\in \Gl_d(R')$ 
have the same order. To see this we note first that
\begin{eqnarray*}
(1+t\alpha(g))^2 &=& 1+ 2t\alpha(g) + t^2\,(\palpha(g)\circ\palpha(g))\quad\mbox{
and }\\ 
(1+t\alpha(g))^{2^i} &=& 1 +2^it\alpha(g)\qquad\mbox{ for $2\le i\le
n$.} 
\end{eqnarray*}
If $g$ has order $1$, the claim is clear from the fact that $\alpha$ is injective.
For $g$ of order $2$ we have $\alpha(g)\in
\M_d(2^{n-1}W/2^nW)$, so $\palpha(g)=0$ and 
the first equality implies $(1 + t\alpha(g))^2 = 1$.  Conversely, if $(1 + t \alpha(g))^2 = 1$,
then this equality implies $2 \alpha(g) = 0$, so $2g$ is the identity because $\alpha$ is
injective.  The second equality and
the injectivity of $\alpha$ similarly imply that $g$ and $1 + t \alpha(g)$ have the same
order if either has order $2^i$ for some $i > 1$.
Next, we remark that the elements $1+t\alpha(g)\in\GL_d(R')$ with
$g$ ranging over $K$ commute with each other because 
the elements of $\palpha(K)$ commute in $M$. Since $K$ is finite and abelian,
we can write $K$ as a direct sum of
cyclic groups.  By setting $$\rho'(g)=1+t\alpha(g)$$ for a generator
of each of these cyclic groups, we  obtain a lift $\rho':K\to \GL_d(R')$ of
the restriction of $\rho_R$ to $K$.

Now suppose that $p=2$ and $n=1$. Our assumption that condition (ii)(b)
fails then provides us with an element $a\in k$ such that for all
$g\in K$ we have $\alpha(g)^2=a\,\alpha(g)$.  Choose an element
$\hat{a} \in W$ with image $a$ in $W/pW = k$, and 
define $$R'=W[[t]]/(2t^2,t^3, 2t+\hat{a}t^2).$$ Then $R = R'/I$ when $I$ 
is the $1$-dimensional vector space
over $k$ generated by $t^2$, and in $R'$ we have $2t=-\hat{a}t^2$ and $2t^2 = 0 = t^3$.
For any $g\in K$ now choose $A_g\in \M_d(R')$ with
$A_g\bmod(p,t) =\alpha(g)$. Then we have
$$
(1+tA_g)^2= 1+2t A_g + t^2 A_g^2 = 1+t^2(-a\alpha(g)+\alpha(g)^2)=1.
$$
Since the elements of $\alpha(K)$ commute in $M=\M_d(k)$ and the element
$t^2$ is annihilated by $(p,t)$ in $R'$,
we also know that the elements $1+t A_g\in \GL_d(R')$
all commute when $g$ ranges over $K$.  So we can find a lift $\rho':K\to \GL_d(R')$ of
the restriction of $\rho_R$ to $K$
by setting $$\rho'(g)=1+tA_g$$ for $g$ in an $\F_2$-basis of $K$.

In summary, in all cases we produced an extension $R'$ of $R$, and 
$$\rho'\colon\; K\to \GL_d(R')$$ lifting the restriction of $\rho_R$ to $K$.
Moreover, $R=R'/I$ where the ideal $I$ is a $k$-vector space of dimension~$1$.
We now claim that we can extend the homomorphism $\rho'$ to $\Gamma$ so that
diagram (\ref{eq:smalllift}) is commutative. Note that
the group $\M_d(I)$ has a conjugation action by
$\Gl_d(R)$, and through $\rho_R$ it has a $\Gamma$-action. This 
$k\Gamma$-module $M_d(I)$ is isomorphic to $M$, so the
obstruction to the existence of a lift $\rho'$ is a class in
$H^2(\Gamma,M)$. 
We know that this lift exists if we restrict to $K$, so the restriction
of this class to $K$ is trivial in $H^2(K,M)$. But by Lemma \ref{lem:cohom} (3)
this restriction map on cohomology groups is injective, so the first class was
trivial as well.  This shows the existence of $\rho'$ lifting $\rho_R$. 

By the deformation ring property, this implies that the morphism $R'\to R$ is
split, i.e., there is a morphism $R\to R'$ so that the composition $R\to R'\to
R$ is the identity.  But then the maps are isomorphisms on the (one-dimensional) cotangent spaces, so
they are also surjective and $R$ is isomorphic to $R'$ which is clearly false.
With this contradiction the proof of Theorem \ref{thm:genresult} is complete.

\section{The inverse problem for $R= W[[t]]/(p^nt,t^2)$}
\label{s:inverse}

In this section, we use Theorem \ref{thm:genresult} to prove Theorem \ref{thm:main1}.
We first establish a special case.

\begin{thm}
\label{thm:twisted}
Suppose $n \ge 1$.
Define 
$$G_0=\mathrm{Gal}(\mathbb{F}_{p^2}/\mathbb{F}_p)\quad\mbox{ and }\quad 
G= \mathbb{F}_{p^2}^{\,*}\semid G_0$$
where $G_0$ acts on the multiplicative group $\mathbb{F}_{p^2}^{\,*}$ by restricting the natural action of $G_0$ on $\mathbb{F}_{p^2}$ to
$\mathbb{F}_{p^2}^{\,*}$. The natural action of  $G_0$ and $\mathbb{F}_{p^2}^{\,*}$  on $V=\mathbb{F}_{p^2}$  makes
$V$ into a projective and simple $\mathbb{F}_p G$-module. 
The endomorphism ring $M=\End_{\mathbb{F}_p}(V)$ is isomorphic to the twisted group ring $\mathbb{F}_{p^2}^{\,\dagger}G_0$
a an $\mathbb{F}_p$-algebra.
There exists a simple projective $\mathbb{F}_pG$-module $V'$ such that
\begin{equation}
\label{eq:v'}
M\cong V'\oplus \mathbb{F}_pG_0
\end{equation}
as $\mathbb{F}_pG$-modules. Let $K $ be a projective $(\mathbb{Z}/p^n)G$-module such that 
$$\mathbb{F}_p\otimes_{\mathbb{Z}/p^n} K \cong V'$$ as $\mathbb{F}_pG$-modules.
Let $\Gamma$ be the semidirect product 
$$\Gamma=K \semid_{\delta} \,G$$ where 
$\delta:G \to \mathrm{Aut}(K)$ is the group homomorphism given by
the $G$-action on the $(\mathbb{Z}/p^n)G$-module $K$.
Let $V$ also stand for the inflation of $V$ to an $\mathbb{F}_p\Gamma$-module.  
Then 
we have a $\mathbb{Z}_p$-algebra isomorphism
$$R_{\mathbb{Z}_p}(\Gamma,V)\cong \mathbb{Z}_p[[t]]/(p^nt, t^2).$$
\end{thm}

\begin{proof}
If $p=2$, then $G$ is isomorphic to the symmetric group $S_3$ on 3 letters and $V$ is 
the unique simple projective $\mathbb{F}_pG$-module, up to isomorphism. If $p\ge 3$, then
the order of $G$ is relatively prime to $p$ and $V$ is also a simple projective $\mathbb{F}_p G$-module.

Since $V=\mathbb{F}_{p^2}$ is a Galois algebra over $\mathbb{F}_p$ with Galois  group $G_0$, it follows that 
$M=\mathrm{End}_{\mathbb{F}_p}(V)$ is canonically isomorphic to the twisted group
ring $\mathbb{F}_{p^2}^{\, \dagger} G_0$ as an $\mathbb{F}_p$-algebra.  
This isomorphism defines an $\mathbb{F}_pG$-module structure on $\mathbb{F}_{p^2}^{\,\dagger}G_0$ by conjugation as follows. Let
$G_0=\langle \sigma\rangle$, let $\mathbb{F}_{p^2}^{\,*}=\langle \zeta\rangle$ and let $x=b_0+b_1\sigma\in \mathbb{F}_{p^2}^{\,\dagger}G_0$,
so $b_0,b_1\in \mathbb{F}_{p^2}$. Then
\begin{eqnarray}
\label{eq:numberthis1}
\sigma . x &=& \sigma x \sigma^{-1} \;= \;(b_0)^p + (b_1)^p\sigma \quad\mbox{ and}\\
\label{eq:numberthis2}
\zeta.x&=&\zeta x \zeta^{-1}\,\;=\;b_0 + b_1\zeta^{1-p}\sigma .
\end{eqnarray}

We have $\mathbb{F}_{p^2}^{\,\dagger}G_0 = \mathbb{F}_{p^2}\oplus \mathbb{F}_{p^2}\sigma$ as $\mathbb{F}_p$-vector spaces. The above $G$-action on $\mathbb{F}_{p^2}^{\,\dagger}G_0$
implies that both $\mathbb{F}_{p^2}$ and $\mathbb{F}_{p^2}\sigma$ are $\mathbb{F}_pG$-submodules of $\mathbb{F}_{p^2}^{\,\dagger}G_0$. It follows for example
from the normal basis theorem that $\mathbb{F}_{p^2}\cong \mathbb{F}_pG_0$ as $\mathbb{F}_pG$-modules, where $\mathbb{F}_{p^2}^{\,*} \subset G$
acts trivially by conjugation on $\mathbb{F}_{p^2}$.  Thus to prove (\ref{eq:v'}) it suffices to show that
$V' = \mathbb{F}_{p^2}\sigma$ is a simple projective $\mathbb{F}_pG$-module.  Since $V$ is a projective $\mathbb{F}_pG$-module, so are $M =\mathrm{End}_{\mathbb{F}_p}(V)\cong 
 \mathbb{F}_{p^2}^{\, \dagger}G_0$ and $V'$. 
Considering the action of $\mathbb{F}_{p^2}^{\,*}=\langle \zeta\rangle$ on $\mathbb{F}_{p^2}\sigma$, we see that the action
of $\zeta$ has eigenvalue $\zeta^{1-p}$. Since $\zeta^{1-p}$ lies in $\mathbb{F}_{p^2}-\mathbb{F}_p$, 
it follows that $V'=\mathbb{F}_{p^2}\sigma$ is a simple projective $\mathbb{F}_pG$-module.  Moreover, 
$$\mathrm{Hom}_{(\mathbb{Z}/p)G}(V', \mathbb{F}_pG_0) = 0$$ 
because $\mathbb{F}_pG_0$ has $\mathbb{F}_p$-dimension $2$
and is not isomorphic to $V'$ since $\mathbb{F}_{p^2}^{\,*}$ acts trivially on $\mathbb{F}_pG_0$.   

For ease of notation we set $W= \Witt(\mathbb{F}_p) = \mathbb{Z}_p$.  
Let $V_{W}$ be a projective $WG$-module such that $\mathbb{F}_p \otimes_{W} V_{W} \cong V$ as $\mathbb{F}_pG$-modules.
Let $M_{W} =\mathrm{End}_{W}(V_{W})$.
To complete the proof, it will suffice to show that $G$, $K$, $M$ and $M_{W}$ satisfy the conditions  in Theorem \ref{thm:genresult}(ii)
when $k = \mathbb{F}_p$.

For all $p$, $K$ is a finitely generated $(\mathbb{Z}/p^n)G$-module.
Define $\Gamma=K \semid_{\delta} \,G$ where 
$\delta:G \to \mathrm{Aut}(K)$ is the group homomorphism given by
the $G$-action on the $(\mathbb{Z}/p^n)G$-module $K$.
Since by our above calculations, $M\ \cong (K/pK)\oplus \mathbb{F}_pG_0$ as 
$(\mathbb{Z}/p)G$-modules, it follows that
\begin{eqnarray*}
\mathrm{Hom}^{cont}(K,M)^G &=& \mathrm{Hom}_{(\mathbb{Z}/p)G}(K/pK,M)\\
& \cong& \mathrm{Hom}_{(\mathbb{Z}/p)G}(V',V'\oplus \mathbb{F}_pG_0)  \\
&\cong &\mathbb{F}_p
\end{eqnarray*}
giving condition (a) of 
Theorem \ref{thm:genresult}(ii). Since $K$ and $M_{W}/p^nM_{W} $ are projective
$(\mathbb{Z}/p^n)G$-modules, it follows that $\mathrm{Hom}_{(\mathbb{Z}/p^n)G}(K,M_{W}/p^n M_{W})$ is a projective
$(\mathbb{Z}/p^n)$-module $H$ such that 
$$H/pH=\mathrm{Hom}_{\mathbb{F}_pG}(K/pK,M)\cong \mathbb{F}_p.$$ Therefore,
$$\mathrm{Hom}_{(\mathbb{Z}/p^n)G}(K,M_{W}/p^n M_{W})\cong \mathbb{Z}/p^n$$ and there exists an injective $(\mathbb{Z}/p^n)G$-module
homomorphism $\alpha:K\to M_{W}/p^n M_{W}$ whose image is not contained in
$pM_{W}/p^n M_{W}$. 
By the above calculations in the
twisted group algebra $\mathbb{F}_{p^2}^{\, \dagger}G_0$, we see that the image of $\alpha$ mod $pM_{W}/p^n M_{W}$ is 
isomorphic to $\mathbb{F}_{p^2}\sigma$. Since for example 
$$(\sigma)\cdot (\zeta\sigma)=\zeta^p\neq
\zeta=(\zeta\sigma)\cdot (\sigma)$$ in the algebra $\mathbb{F}_{p^2}^{\, \dagger}G_0 \cong M $, we obtain that the image of $\alpha$ mod $pM_{W}/p^n M_{W}$ is not
commutative with respect to the multiplication in the ring $M_{W}/p^n M_{W}$.
This gives condition (b) of Theorem \ref{thm:genresult}(ii) and completes the proof. 
\end{proof}

\begin{rem}
\label{rem:s3worksalmostalways}
If $p>3$, we can replace the group $G$ in Theorem \ref{thm:twisted} by
the symmetric group $S_3$ and $V$ by a 2-dimensional simple projective 
$\mathbb{F}_p S_3$-module (which is unique up to isomorphism). It follows then that 
$$M=\mathrm{Hom}_{\mathbb{F}_p}(V,V)\cong
\mathbb{F}_p[\mathbb{Z}/2]\oplus V$$ as $\mathbb{F}_pG$-modules, which means that we can take
$V'=V$  in this case.
\end{rem}

\begin{rem}
\label{rem:negative}
As mentioned in the introduction, in subsequent work on Question \ref{q:inverse}, 
Rainone proved in \cite{R} that if $p > 3$ and $1 \le m \le n$, the ring 
$\mathbb{Z}_p[[t]]/(p^n,p^m t)$ is a universal deformation ring relative to 
$W = \mathbb{Z}_p$.  These rings and the rings of Theorems \ref{thm:main1} and 
\ref{thm:twisted} form disjoint sets of isomorphism classes.  
Rainone's work gave the first negative answers to two questions of  Bleher and 
Chinburg (Question 1.2 of \cite{lcann} and Question 1.1 of \cite{BC}).  Later we observed that 
Theorem \ref{thm:twisted} also gives a negative answer to Question 1.2 of \cite{lcann} when 
$p > 2$. 
\end{rem}

\medbreak
\noindent {\bf Completion of the  Proof of Theorem \ref{thm:main1}.}
Let $k$, $p$, $W$ and $n$ be as in Theorem \ref{thm:main1}.
By Theorem \ref{thm:twisted}, there is a finite group $\Gamma$ and
a representation $V_0$ of $\Gamma$ over $\mathbb{F}_p$ such that
$\mathrm{End}_{\mathbb{F}_p G}(V_0) = \mathbb{F}_p$ and
the universal deformation ring $R_{\mathbb{Z}_p}(\Gamma,V_0)$ is isomorphic to
$\mathbb{Z}_p[[t]]/(p^n t, t^2)$.  Let $V = k \otimes_{\mathbb{F}_p} V_0$.
Then $$\mathrm{End}_{kG}(V) \cong k \otimes_{\mathbb{F}_p} \mathrm{End}_{\mathbb{F}_p G}(V_0) 
\cong k.$$
By Theorem \ref{thm:basechange}, the universal deformation ring $R_{W}(\Gamma,V)$
is isomorphic to the completion of $W \otimes_{\mathbb{Z}_p} \mathbb{Z}_p[[t]]/(p^n t, t^2)$
with respect to its maximal ideal.  This completion is isomorphic to 
$W[[t]]/(p^n t, t^2)$.  It remains to show that this ring is not a 
complete intersection if $p^n W \ne \{0\}$.  This is clear if $W$ is regular.
In general, if one assumes that $W[[t]]/(p^n t, t^2)$ is a complete
intersection, then $W$ is a quotient $S/I$ for some regular complete
local commutative Noetherian ring $S$ and a proper ideal $I$ of $S$.  If $S'=S[[t]]$, then 
$$W[[t]]/(p^n t, t^2) = S'/I'$$ when $I'$ is the ideal of $S'$ generated
by $I$, $p^n t$ and $t^2$.   Since 
$$\mathrm{dim}\, W[[t]]/(p^nt,t^2) = \mathrm{dim}\,W,$$
we obtain by \cite[Thm. 21.1]{matsumura} that
\begin{eqnarray}
\label{eq:equal}
\mathrm{dim}_k(I'/m_{S'}I') &=& \mathrm{dim}\, S' -\mathrm{dim}\, (S'/I')\\
&=& \mathrm{dim}\,S +1 - \mathrm{dim}\, (S/I) \nonumber\\
&\le& \mathrm{dim}_k(I/m_SI)+1.\nonumber
\end{eqnarray}
Using power series expansions, we see that 
$$\mathrm{dim}_k(I'/m_{S'}I') =\mathrm{dim}_k(I/m_SI)+2$$ if $p^n W \ne \{0\}$.
Since this contradicts (\ref{eq:equal}), $W[[t]]/(p^nt,t^2)$ is not a 
complete intersection if $p^nW\neq\{0\}$. This completes the proof of Theorem
\ref{thm:main1}.

\begin{rem}
\label{rem:referee}
The referee of this paper asked whether in our examples of universal deformation rings
$R_W(\Gamma,V)$ which are not complete intersections, we have
$H^3(\Gamma,\mathrm{End}_k(V))=0$. 
Consider, for example, the case when $n=1$ in Theorem \ref{thm:twisted} so that 
$k=\mathbb{F}_p$.  We claim that $H^3(\Gamma,\mathrm{End}_{\mathbb{F}_p}(V))\neq 0$.

To see this, recall that
$G= \mathbb{F}_{p^2}^{\,*}\semid G_0$ where $G_0=\mathrm{Gal}(\mathbb{F}_{p^2}/
\mathbb{F}_p)$ and $G_0$ acts on $\mathbb{F}_{p^2}^{\,*}$ using the natural action of 
$G_0$ on $\mathbb{F}_{p^2}$. The $\mathbb{F}_pG$-module $V$ is defined to be
$V=\mathbb{F}_{p^2}$ with the natural actions of  $G_0$ and $\mathbb{F}_{p^2}^{\,*}$. 
By (\ref{eq:v'}), we have
$$\End_{\mathbb{F}_p}(V)\cong V'\oplus \mathbb{F}_pG_0$$
for a simple projective $\mathbb{F}_pG$-module $V'$. Letting $K=V'$, we then have
$\Gamma=K \semid_{\delta} \,G$ where 
$\delta:G \to \mathrm{Aut}(K)$ is given by the $G$-action on $K=V'$.
Let $\Gamma_0= K \semid_{\delta} \,\mathbb{F}_{p^2}^{\,*}$. 
Using that $\mathbb{F}_pG_0\cong \mathrm{Ind}_{\Gamma_0}^\Gamma 
\mathbb{F}_p$ where $\mathbb{F}_p$ is the trivial simple $\mathbb{F}_p\Gamma_0$-module, 
we obtain
$$\mathrm{End}_{\mathbb{F}_p}(V)\cong V' \oplus \mathrm{Ind}_{\Gamma_0}^\Gamma 
\mathbb{F}_p.$$
By Frobenius reciprocity, we have
$$H^3(\Gamma,\mathrm{Ind}_{\Gamma_0}^\Gamma \mathbb{F}_p)\cong 
H^3(\Gamma_0, \mathbb{F}_p).$$
Using the Kummer sequence $1\to \mu_p\to \mathbb{C}^*\xrightarrow{p}\mathbb{C}^*\to 1$
where $\mu_p$ is the subgroup of  $p^{\mathrm{th}}$ roots of unity in $\mathbb{C}^*$,
we see that the quotient  group $H^2(\Gamma_0,\mathbb{C}^*)/p\cdot 
H^2(\Gamma_0,\mathbb{C}^*)$ is isomorphic to a subgroup of $H^3(\Gamma_0, \mathbb{F}_p)$. 

If $p=2$ then $\Gamma=S_4$ and $\Gamma_0=A_4$, and we have
$H^2(A_4, \mathbb{C}^*)
\cong \mathbb{Z}/2$, which implies $H^3(\Gamma,\mathrm{End}_{\mathbb{F}_p}(V))\neq 0$
when $p=2$.

Suppose now that $p\ge 3$. Since
$H^1(\Gamma_0,\mathbb{C}^*)=\mathrm{Hom}(\Gamma_0,\mathbb{C}^*)$
and the maximal abelian quotient of $\Gamma_0$ is isomorphic to $\mathbb{F}_{p^2}^{\,*}$, 
it follows that $H^2(\Gamma_0,\mathbb{F}_p)$
injects into $H^2(\Gamma_0,\mathbb{C}^*)$. Therefore, to show that 
$H^3(\Gamma,\mathrm{End}_{\mathbb{F}_p}(V))\neq 0$, it suffices to prove that
$H^2(\Gamma_0,\mathbb{F}_p)\neq 0$. Using the short exact
sequence
$$1\to K \to \Gamma_0\to \mathbb{F}_{p^2}^{\,*}\to 1$$
and that $H^i(K,\mathbb{F}_p)$ is cohomologically trivial as a module for $\mathbb{F}_{p^2}^{\,*}$
for all $i\ge 0$,
we see that
$H^2(\Gamma_0,\mathbb{F}_p)\cong H^0( \mathbb{F}_{p^2}^{\,*},H^2(K,\mathbb{F}_p))$.
Consider the function $c:K\times K\to \bigwedge^2 K$ given by the wedge product.
It follows from the bilinearity of the wedge product that $c$ is a 2-cocycle. Since $p\ge 3$,
the anti-commutativity of the wedge product implies  that $c$ is not a 2-coboundary.
Furthermore, $c$ is an invariant 2-cocycle with respect to the action of $\mathbb{F}_{p^2}^{\,*}$
on $K\times K$ and on $\bigwedge^2 K$. Here $\bigwedge^2 K$ can be identified with
$\mathbb{F}_p$ with trivial action by $\mathbb{F}_{p^2}^{\,*}$ in view of (\ref{eq:numberthis2})
since $K\cong \mathbb{F}_{p^2}\sigma$ as an $\mathbb{F}_pG$-module.
It follows that $c$ defines a non-zero element in 
$H^0( \mathbb{F}_{p^2}^{\,*},H^2(K,\mathbb{F}_p))$.
Therefore, we obtain $H^3(\Gamma,\mathrm{End}_{\mathbb{F}_p}(V))\neq 0$
for all primes $p$.
\end{rem}

\begin{rem}
\label{rem:moregeneral}
To construct more examples to which Theorem \ref{thm:genresult} applies, there 
are two fundamental issues.  One must construct a group $G$ and a projective
$kG$-module $V$ for which both the left $kG$-module structure and the ring structure 
of $M = \mathrm{Hom}_k(V,V)$
can be analyzed sufficiently well to be able to produce a 
$G$-module $K$ having the properties in the Theorem. When one can identify
the ring $\mathrm{Hom}_k(V,V)$ with a twisted group algebra, as in the proof of
Theorem \ref{thm:twisted}, this can be very useful in checking condition
(ii)(b) of Theorem \ref{thm:genresult}.  A natural approach to analyzing
the $kG$-module structure of $M$ is to note that the Brauer character $\xi_M$
of $M$ is the tensor product $\xi_V \otimes \xi_{V^*}$ of the Brauer characters
of $V$ and its $k$-dual $V^*$.   For example, if $V$ is induced from a 
representation $X$ of a subgroup $H$ of $G$, then  $\xi_V$ is given by 
the usual formula for the character  of an induced representation.  If $\mathrm{dim}_k(X) = 1$,
the analysis
of the ring structure of $M$ becomes a combinatorial problem using $X$ and coset representatives
of $H$ in $G$.  
We carry out this program with some further examples in the next section.
\end{rem}

\section{Further examples}
\label{s:examples}

In all the examples for Theorem \ref{thm:main1} which were discussed in \S \ref{s:inverse}, 
the dimension of $V$ is 2. 
In this section, we weaken this restriction on the dimension.
Let $W$ be a complete Noetherian local commutative ring with residue field $k$ of
positive characteristic $p$. We prove the following result.

\begin{thm}
\label{thm:alldegrees}
Suppose $d\ge 2$ is an integer such that either $d<p-1$ or $d=p^f$ for some positive integer $f$.
There exists a profinite group $\Gamma$ 
satisfying Hypothesis \ref{hyp:finite}  and a continuous representation $V$ of $\Gamma$ over $k$
of degree $d$ such that the versal deformation ring $R_W(\Gamma,V)$ is isomorphic to
$W[[t]]/(t^2,pt)$.
\end{thm}

\begin{proof}
Because of Theorem \ref{thm:basechange} it suffices to prove Theorem \ref{thm:alldegrees} 
when $k=\mathbb{F}_p$ and  $W=\Witt(\mathbb{Z}_p)$, which allows us to use Theorem 
\ref{thm:genresult}.

Let $\Omega=\{1,2,\ldots,d+1\}$, and let $S_{d+1}$ be the symmetric group on $\Omega$.
For each prime power $q=p^f$, the projective general linear group $\PGL_2(\mathbb{F}_q)$
acts faithfully and triply transitively on the projective line
$\mathbb{P}^1({\mathbb{F}_q})$ and is thus
isomorphic to a subgroup of $S_{q+1}$.
If $d<p-1$ let $G=S_{d+1}$, and if $d=q=p^f\ge 2$ let $G=\PGL_2(\mathbb{F}_q)$
viewed as a subgroup of $S_{d+1}=S_{q+1}$.

Consider the natural $(d+1)$-dimensional 
representation $N$ of $S_{d+1}$ with $\mathbb{F}_p$-basis 
$\{b_1,\ldots, b_{d+1}\}$ such that for all $\sigma\in S_{d+1}$ and all $i\in\Omega$, 
$\sigma . b_i = b_{\sigma(i)}$. On restricting to $G$, $N$ is also a $(d+1)$-dimensional
representation of $G$ over $\mathbb{F}_p$. Let $T$ be the 1-dimensional 
$\mathbb{F}_p$-subspace of $N$ generated by $(b_1+\cdots + b_{d+1})$, and let
$V$ be the $d$-dimensional $\mathbb{F}_p$-subspace of $N$ with $\mathbb{F}_p$-basis
$$\{b_1-b_2,b_2-b_3,\ldots,b_d-b_{d+1}\}.$$
Then $T$ and $V$ are $\mathbb{F}_pG$-submodules of $N$, where $G$ acts trivially on $T$.
In fact, since $p\nmid (d+1)$ we have $N=T\oplus V$. Since $T$ and $V$ can be lifted to
$\mathbb{Z}_pG$-modules $\hat{T}$ and $\hat{V}$ which are free over $\mathbb{Z}_p$
and since $G$ acts doubly transitively on $\Omega$, it follows that
$\mathbb{Q}_p\otimes_{\mathbb{Z}_p}\hat{V}$ is an absolutely irreducible representation
of $G$ over $\mathbb{Q}_p$ (see e.g. \cite[Ex. 9 on p. 877]{dummitfoote}). 
If $d<p-1$ then the order of $G=S_{d+1}$ is relatively prime to $p$
which immediately implies that $V$ is a simple and projective $\mathbb{F}_pG$-module.
If $d=q$ and $G=\mathrm{PGL}_2(\mathbb{F}_q)$ then the Sylow $p$-subgroups of
$G$ have order $q$ which implies by \cite[Prop. 46 on p. 136]{SerreRep} that 
$V=\hat{V}/p\hat{V}$ is a simple and projective $\mathbb{F}_pG$-module.
Note that in this case, the character of $G$ corresponding to $\hat{V}$ is the Steinberg
character of $\mathrm{PGL}_2(\mathbb{F}_q)$.

Let $K=V$ and let $\delta:G\to\mathrm{Aut}(K)$ be the group homomorphism given by the 
action of $G$ on $K=V$. Define $\Gamma = K\semid_\delta G$ and view $V$ also as
an $\mathbb{F}_p\Gamma$-module via inflation.
Let $M=\mathrm{End}_{\mathbb{F}_p}(V)$. 
To prove Theorem \ref{thm:alldegrees}, it suffices to prove
that $G$, $K$ and $M$ satisfy the conditions in Theorem \ref{thm:genresult}(ii) when 
$k=\mathbb{F}_p$ and $n=1$.

For an $\mathbb{F}_pG$-module $X$, let $X^*$ denote its $\mathbb{F}_p$-dual.
Consider 
$$\mathrm{End}_{\mathbb{F}_p}(N)\cong N^*\otimes_{\mathbb{F}_p} N\cong 
N\otimes_{\mathbb{F}_p} N.$$
We can identify $\mathrm{End}_{\mathbb{F}_p}(N)$ with 
$\mathrm{Mat}_{d+1}({\mathbb{F}_p})$, where we identify $b_i\otimes b_j$ with the
elementary matrix $E_{i,j}$ which has coefficient 1 at position $(i,j)$ and coefficient 0 otherwise.
Since $M=\mathrm{End}_{\mathbb{F}_p}(V)\cong V^*\otimes_{\mathbb{F}_p}V
\cong V\otimes_{\mathbb{F}_p} V$,
$M$ can be identified with the subspace of $\mathrm{End}_{\mathbb{F}_p}(N)$ with 
${\mathbb{F}_p}$-basis $\{D_{i,j}\;|\; 1\le i,j\le d\}$, where
$$D_{i,j} = E_{i,j}-E_{i,j+1}-E_{i+1,j}+E_{i+1,j+1}$$
for all $1\le i,j\le d$. 
Note that $\mathrm{End}_{\mathbb{F}_p}(N)$ and $M$ are 
${\mathbb{F}_p}S_{d+1}$-modules by letting $S_{d+1}$ act by conjugation, i.e. 
$\sigma. E_{i,j} = E_{\sigma(i),\sigma(j)}$ for all $\sigma\in S_{d+1}$. In particular, it follows
that $\mathrm{End}_{\mathbb{F}_p}(N)$ and $M$ are ${\mathbb{F}_p}G$-modules.

We now show that the irreducible representation $V$ occurs with multiplicity $1$ in~$M$.
For ${\mathbb{F}_p}G$-modules $X$ and $Y$, let 
$$\langle X,Y\rangle=\mathrm{dim}_{\mathbb{F}_p}\mathrm{Hom}_{\mathbb{F}_pG}(X,Y)
=\mathrm{dim}_{\mathbb{F}_p}\left(X^*\otimes_{\mathbb{F}_p} Y\right)^G.$$
Since $V$ is a simple projective $\mathbb{F}_pG$-module which is isomorphic to
its $\mathbb{F}_p$-dual $V^*$, it follows that
the multiplicity of $V$ in $M$ is equal to 
$$\langle V,M\rangle=\mathrm{dim}_{\mathbb{F}_p}\left(
V\otimes_{\mathbb{F}_p}V\otimes_{\mathbb{F}_p}V\right)^G.$$
Using that $N=T\oplus V$, we obtain that 
$$N\otimes_{\mathbb{F}_p}N\otimes_{\mathbb{F}_p}N
\cong T \;\oplus\;  V^3 \; \oplus \; (V\otimes_{\mathbb{F}_p}V)^3\;\oplus\;
(V\otimes_{\mathbb{F}_p}V\otimes_{\mathbb{F}_p}V).$$
Since $V^G=0$ and $\mathrm{dim}_{\mathbb{F}_p}\left(V\otimes_{\mathbb{F}_p}V\right)^G
=\langle V, V\rangle = 1$, it follows that
\begin{equation}
\label{eq:vmult}
\langle V,M\rangle = \mathrm{dim}_{\mathbb{F}_p}\left(
N\otimes_{\mathbb{F}_p}N\otimes_{\mathbb{F}_p}N\right)^G - 4.
\end{equation}
Note that $N\otimes_{\mathbb{F}_p}N\otimes_{\mathbb{F}_p}N$ is the representation of $G$
corresponding to the diagonal action of $G$ on the set $\Omega$ of
all triples $(a,b,c)$ with $a,b,c\in\{1,\ldots,d+1\}$.
This implies that $\mathrm{dim}_{\mathbb{F}_p}\left(
N\otimes_{\mathbb{F}_p}N\otimes_{\mathbb{F}_p}N\right)^G$ is equal to the number of
$G$-orbits on $\Omega$. Since $G$ acts triply transitively on $\Omega$, this number 
is equal to 5. By (\ref{eq:vmult}), it follows that
$\langle V,M \rangle = 1$, i.e. $V$ occurs with multiplicity 1 in $M$. Therefore, 
condition (ii)(a) of Theorem \ref{thm:genresult} is satisfied for $K=V$ when $n=1$.

Let $\rho$ be the $(d+1)$-cycle $\rho=(1,2,\cdots,d+1)\in S_{d+1}$, and define
$$x_1=\left(\begin{array}{ccrcr}
d-1&0&-1&\cdots&-1\\ 0&-d+1\;\;&1&\cdots&1\\
-1\;\,\,&1&0&\cdots&0\\ \vdots&\vdots&\vdots&&\vdots\\ -1\;\,\,&1&0&\cdots&0
\end{array}\right).$$
Since $x_1=(d-1)D_{1,1} + \sum_{\ell =2}^{d} (d+1-\ell)\,(D_{1,\ell}+D_{\ell,1})$,
we see that $x_1\in M$.
Define
$$x_j=\rho^{j-1}. x_1$$
for $2\le j\le d$. An easy matrix calculation shows that the subspace of $M$ with 
$\mathbb{F}_p$-basis $\{x_1,\ldots, x_d\}$ is an ${\mathbb{F}_p}S_{d+1}$-submodule 
of $M$ which is isomorphic to $V$. 
On restricting the action to $G$, we see that the additive group homomorphism
$\alpha:K=V\to M$, defined by $\alpha(b_j-b_{j+1}) = x_j$ for $1\le j\le d$, is an injective 
$G$-equivariant homomorphism.
Since $x_1x_2\neq x_2x_1$ in $M\subset\mathrm{Mat}_d(\mathbb{F}_p)$, 
condition (ii)(b) of Theorem \ref{thm:genresult} is satisfied when $n=1$. 
This completes the proof of Theorem \ref{thm:alldegrees}.
\end{proof}

\section{Appendix: Proof of Theorem $\ref{thm:basechange}$}
\label{s:appendix}

We assume the notation in the statement of Theorem \ref{thm:basechange}. 
Let $R=R_{W}(\Gamma,\rho)$.
Recall that $\Omega=W'\otimes_{W}R$ and $R'$ is the completion
of $\Omega$ with respect to its unique maximal ideal $\mathfrak{m}_\Omega$.
Define $\hat{\mathcal{C}}'$ to be the category of all 
complete local commutative Noetherian $W'$-algebras with residue field $k'$.
Let $\nu:\Gamma\to\mathrm{GL}_d(R)$ be a versal lift of $\rho$ over $R$,
and let $$\nu':\Gamma\to\mathrm{GL}_d(R')$$ be the lift of $\rho'$ over $R'$ defined by
$$\nu'(g)=\left(1\otimes\nu(g)_{i,j}\right)_{1\le i,j\le d}$$ for all $g\in \Gamma$.

The first step is to show that if $A' \in \mathrm{Ob}(\mathcal{C}')$ is an
Artinian $W'$-algebra with residue field $k'$ and
$\tau':\Gamma\to\mathrm{GL}_d(A')$ is a lift of $\rho'$ over $A'$, then there is a morphism 
$\alpha:R'  \to A'$ in $\hat{\mathcal{C}}'$ such that $$[\tau']=[\alpha\circ\nu'].$$
Since $A'$ is Artinian,
$\mathrm{Hom}_{\hat{\mathcal{C}}'}(R',A')$ is equal to
the space $\mathrm{Hom}_{\mathrm{cont}}(\Omega,A')$ of continuous $W'$-algebra
homomorphisms which induce the identity map on the residue field $k'$.
Because of Hypothesis \ref{hyp:finite}, one can find a finite set $S\subseteq\Gamma$ such that
$\tau'(S)$ is a set of topological generators for the image of $\tau'$.
Since $\rho'$ and $\rho$ have the same image in $\mathrm{GL}_d(k)\subset\mathrm{GL}_d(k')$,
there exists for each $g\in S$ a matrix $t(g) \in \mathrm{Mat}_d(W)$  such that all entries of
the matrix $\tau'(g) - t(g)$ lie in the maximal ideal $\mathfrak{m}_{A'}$ of $A'$. 
Let $T\subseteq \mathfrak{m}_{A'}$ be the finite set of all matrix entries of $\tau'(g) - t(g)$ as 
$g$  ranges over $S$. Then there is a continuous homomorphism $$f:W[[x_1,\ldots,x_m]] \to A'$$
with $m = \# T$ and $\{f(x_i)\}_{i = 1}^m = T$.   Since $A'$ has the discrete topology, the image
$B$ of $f$ must be a local Artinian $W$-algebra with residue field $k$.
Since $\tau'(S)$ is a set of topological generators for the image of $\tau'$, it follows that
$\tau'$ defines a lift of $\rho$ over $B$. Because   $\nu:\Gamma\to\mathrm{GL}_d(R)$ 
is a versal lift of $\rho$ over the versal deformation ring $R = R_{W}(\Gamma,\rho)$ of $\rho$,
there is a  morphism 
$\beta:R  \to B$ in $\hat{\mathcal{C}}$ such that $\tau':\Gamma \to \mathrm{GL}_d(B)$
is conjugate to $\beta \circ \nu$ by a matrix in the kernel of
$$\pi_B:\mathrm{GL}_d(B) \to \mathrm{GL}_d(B/\mathfrak{m}_B) = \mathrm{GL}_d(k).$$
Let $\beta':R \to A'$ be the composition of $\beta$ with the inclusion $B \subset A'$.
Define $\alpha: R'  \to A'$
to be the morphism in $\hat{\mathcal{C}}'$ corresponding 
to the continuous $W'$-algebra homomorphism 
$\Omega=W' \otimes_{W} R  \to A'$ which sends $w'\otimes r$ to $w'\cdot\beta'(r)$
for all $w'\in W'$ and $r\in R$.
It follows that $\alpha$ satisfies $[\tau']=[\alpha\circ\nu']$.

The second step is to show that  when $k'[\epsilon]$ is the ring of dual numbers over
$k'$, then $\mathrm{Hom}_{\hat{\mathcal{C}}'}(R',k'[\epsilon])$
is canonically identified with the set $\mathrm{Def}_{\rho'}(k'[\epsilon])$ of deformations of $\rho'$ over
$k'[\epsilon]$.  Since $k'[\epsilon]$ is Artinian, it suffices to show that 
$\mathrm{Hom}_{\mathrm{cont}}(\Omega,k'[\epsilon])$ is 
identified with $\mathrm{Def}_{\rho'}(k'[\epsilon])$.  Let 
\begin{eqnarray*}
T(W',\Omega) &=& \frac{\mathfrak{m}_{\Omega}}{\mathfrak{m}_{\Omega}^2 + \Omega \cdot \mathfrak{m}_{W'}}
\quad\mbox{ and}\\
T(W,R) &=& \frac{\mathfrak{m}_R}{\mathfrak{m}_R^2 + R \cdot \mathfrak{m}_{W}}
\end{eqnarray*}
so that we have natural isomorphisms 
\begin{eqnarray*}
\mathrm{Hom}_{\mathrm{cont}}(\Omega,k'[\epsilon]) &\cong& 
\mathrm{Hom}_{k'}(T(W',\Omega),k')\quad\mbox{ and}\\
\mathrm{Hom}_{\hat{\mathcal{C}}}(R,k[\epsilon]) &\cong& \mathrm{Hom}_{k}(T(W,R),k).
\end{eqnarray*}
Since $\mathrm{Ad}(\rho') = k' \otimes_k \mathrm{Ad}(\rho)$, we have from \cite[Prop. 21.1]{maz2} 
that there are natural isomorphisms
\begin{eqnarray*}
\mathrm{Def}_{\rho'}(k'[\epsilon]) &=& H^1(\Gamma,\mathrm{Ad}(\rho')) \\
&=& k' \otimes_k H^1(\Gamma,\mathrm{Ad}(\rho))\\
&=&k'\otimes_k \mathrm{Def}_{\rho}(k[\epsilon]).
\end{eqnarray*}
Hence it suffices to show that the natural homomorphism 
$$\mu:k' \otimes_k  T(W,R) \to T(W',\Omega)$$ is an isomorphism of $k'$-vector spaces.  
Since $\mathfrak{m}_{W}$ is finitely generated, one can reduce to the case when
$W = k$, by considering generators $\alpha$ of $\mathfrak{m}_W$ and 
successively replacing 
$W$ by $W/(W\alpha)$ and $R$ by  $R/(R\alpha)$.
One then  divides $W'$ and $\Omega$ further by ideals generated by generators 
for $\mathfrak{m}_{W'}$ to be able to assume that $W' = k'$.  
However, the case when $W = k$ and $W' = k'$ is obvious, since then
\begin{eqnarray*}
T(k',\Omega)& =& \mathfrak{m}_{\Omega}/\mathfrak{m}_{\Omega}^2 \\
&\cong &k' \otimes_k 
\left(\mathfrak{m}_{R}/\mathfrak{m}_{R}^2\right)\\
&=&  k' \otimes T(k,R).
\end{eqnarray*}
This completes the proof of Theorem \ref{thm:basechange}.

\end{document}